\documentclass[11pt]{article}
\usepackage{amsmath, amscd, amssymb, latexsym, epsfig, color, amsthm}
\usepackage{verbatim}
\usepackage{subfig}
\numberwithin{equation}{section}

\newtheorem{theorem}{Theorem}[section]
\newtheorem{proposition}[theorem]{Proposition}

\newtheorem{corollary}[theorem]{Corollary}

\newtheorem{lemma}[theorem]{Lemma}

\theoremstyle{definition}
\newtheorem{definition}[theorem]{Definition}

\newtheorem{remark}[theorem]{Remark}

\DeclareMathOperator{\skel}{Skel}
\DeclareMathOperator\lk{\mathrm{lk}}
\DeclareMathOperator\st{\mathrm{st}}

\DeclareMathOperator{\conv}{\mathrm{conv}}
\DeclareMathOperator\crtr{\mathrm{crtr}}
\DeclareMathOperator\sd{\mathrm{sd}}
\DeclareMathOperator\so{\mathrm{so}}
\DeclareMathOperator{\dist}{\mathrm{dist}}

\newcommand{\field}{{\bf k}}

\newcommand{\R}{{\mathbb R}}

\title{A characterization of homology manifolds with $g_2\leq 2$}
\author{Hailun Zheng\\
	\small Department of Mathematics \\[-0.8ex]
	\small University of Washington\\[-0.8ex]
	\small Seattle, WA 98195-4350, USA\\[-0.8ex]
	\small \texttt{hailunz@math.washington.edu}
}
\begin{document}
	\maketitle
	\begin{abstract}
		We characterize homology manifolds with $g_2\leq 2$. Specifically, using retriangulations of simplicial complexes, we give a short proof of Nevo and Novinsky's result on the characterization of homology $(d-1)$-spheres with $g_2=1$ for $d\geq 5$ and extend it to the class of normal pseudomanifolds. We proceed to prove that every prime homology manifold with $g_2=2$ is obtained by centrally retriangulating a polytopal sphere with $g_2\leq 1$  along a certain subcomplex. This implies that all homology manifolds with $g_2=2$ are polytopal spheres.
	\end{abstract}
	\section{Introduction}
	Characterizing face-number related invariants of a given class of simplicial complexes has been a central topic in topological combinatorics. One of the most well-known results is the $g$-theorem (see \cite{BL}, \cite{BL2}, and \cite{St}), which completely characterizes the $g$-vectors of simplicial $d$-polytopes. It follows from the $g$-theorem that for every simplicial $d$-polytope $P$, the $g$-numbers of $P$, $g_0,g_1,\cdots, g_{\left\lfloor d/2\right\rfloor}$, are non-negative. This naturally leads to the question of when equality $g_{i}=0$ is attained for a fixed $i$. While it is easy to see that $g_1(P)=0$ holds if and only if $P$ is a $d$-dimensional simplex, the question of which polytopes satisfy $g_2=0$ is already highly non-trivial. This question was settled by Kalai \cite{K}, using rigidity theory of frameworks, in the generality of simplicial manifolds; his result was then further extended by Tay \cite{T} to all normal pseudomanifolds. 
	
	To state these results, known as the lower bound theorem, recall that a \emph{stacking} is the operation of building a shallow pyramid over a facet of a given simplicial polytope, and a \emph{stacked $(d-1)$-sphere} on $n$ vertices is the $(n-d)$-fold connected sum of the boundary complex of a $d$-simplex, denoted as $\partial \sigma^d$, with itself.
	\begin{theorem}\label{thm: Kalai}
		Let $\Delta$ be a normal pseudomanifold of
		dimension $d\geq 3$. Then $g_2(\Delta)\geq 0$. Furthermore, if $d\geq 4$, then equality holds if and only if $\Delta$ is a stacked sphere.
	\end{theorem}
	Continuing this line of research, Nevo and Novinsky \cite{NN} characterized all homology spheres with $g_2=1$. Their main theorem is quoted below.
	\begin{theorem}\label{thm: NevoNovinsky}
		Let $d \geq 4$, and let $\Delta$ be a homology $(d-1)$-sphere without missing facets. Assume that $g_2(\Delta) = 1$. Then $\Delta$ is combinatorially isomorphic to either the join of $\partial \sigma^i$ and $\partial \sigma^{d-i}$, where $2\leq i\leq d-2$, or the join of $\partial \sigma^{d-2}$ and a cycle. Hence every homology $(d-1)$-sphere with $g_2=1$ is combinatorially isomorphic to a homology $(d-1)$-sphere obtained by
		stacking over any of these two types of spheres.
	\end{theorem}
	Their result implies that all homology spheres with $g_2=1$ are polytopal. The proof is based on rigidity theory for graphs.
	
	In this paper, we characterize all homology manifolds with $g_2\leq 2$. Our main strategy is to use three different retriangulations of simplicial complexes  with the properties that (1) the homeomorphism type of the complex is preserved under these retriangulations; and (2), the resulting changes in $g_2$ are easy to compute. Specifically, for a large subclass of these retriangulations, $g_2$ increases or decreases exactly by one. We use these properties to show that every homology manifold with $g_2\leq 2$ is obtained by centrally retriangulating a polytopal sphere of the same dimension but with a smaller $g_2$. As a corollary, every homology sphere with $g_2\leq 2$ is polytopal. Incidentally, this implies a result of Mani \cite{Ma} that all triangulated spheres with $g_1\leq 2$ are polytopal.
	
	This paper is organized as follows. In Section 2 we recall basic definitions and results pertaining to simplicial complexes, polytopes and framework rigidity. In Section 3 we define three retriangulations of simplicial complexes that serve as the main tool in later sections. In Section 4 and Section 5 we use these retriangulations to characterize normal pseudomanifolds with $g_2=1$ (of dimension at least four) and homology manifolds with $g_2=2$ (of dimension at least three), respectively, see Theorems \ref{thm: g_2=1 normal pseudomanifold}, \ref{thm:g_2=2 and d>4} and \ref{thm: g_2=2 and d=4}.
	\section{Preliminaries}
	\subsection{Basic definitions}
    We begin with basic definitions. A \textit{simplicial complex} $\Delta$ on vertex set $V=V(\Delta)$ is a collection of subsets $\sigma\subseteq V$, called \textit{faces}, that is closed under inclusion, and such that for every $v \in V$, $\{v\} \in \Delta$. The \textit{dimension} of a face $\sigma$ is $\dim(\sigma)=|\sigma|-1$, and the \textit{dimension} of $\Delta$ is $\dim(\Delta) = \max\{\dim(\sigma):\sigma\in \Delta\}$. The \textit{facets} of $\Delta$ are maximal faces of $\Delta$ under inclusion. We say that a simplicial complex $\Delta$ is \textit{pure} if all of its facets have the same dimension. A \textit{missing} face of $\Delta$ is any subset $\sigma$ of $V(\Delta)$ such that $\sigma$ is not a face of $\Delta$ but every proper subset of $\sigma$ is. A missing $i$-face is a missing face of dimension $i$. A pure simplicial complex $\Delta$ is \emph{prime} if it does not have any missing facets.
		
	The \textit{link} of a face $\sigma$ is $\lk_\Delta \sigma:=\{\tau-\sigma\in \Delta: \sigma\subseteq \tau\in \Delta\}$, and the \textit{star} of $\sigma$ is $\st_\Delta \sigma:= \{\tau \in\Delta: \sigma\cup\tau\in\Delta \}$. If $W\subseteq V(\Delta)$ is a subset of vertices, then we define the \emph{restriction} of $\Delta$ to $W$ to be the subcomplex $\Delta[W]=\{\sigma\in \Delta: \sigma\subseteq W\}$. The \emph{antistar} of a vertex is ${\rm{ast}}_\Delta(v)=\Delta[V-\{v\}]$. We also define the $i$-\emph{skeleton} of $\Delta$, denoted as $\skel_i(\Delta)$, to be the subcomplex of all faces of $\Delta$ of dimension at most $i$. If $\Delta$ and $\Gamma$ are two simplicial complexes on disjoint vertex sets, their \emph{join} is the simplicial complex $\Delta*\Gamma=\{\sigma\cup\tau:\sigma\in\Delta, \tau\in\Gamma\}$. When $\Delta$ consists of a single vertex, we write the \emph{cone} over $\Gamma$ as $u*\Gamma$. 
		
	A \emph{polytope} is the convex hull of a finite set of points in some $ \R^e$. It is called a $d$-polytope if it is $d$-dimensional. A polytope is \emph{simplicial} if all of its facets are simplices. A \emph{simplicial sphere} (resp. ball) is a simplicial complex whose geometric realization is homeomorphic to a sphere (resp. ball). The boundary complex of a simplicial polytope is called a \emph{polytopal sphere}. We usually denote the $d$-simplex by $\sigma^d$ and its boundary complex by $\partial \sigma^d$. For a fixed field $\field$, we say that $\Delta$ is a $(d-1)$-dimensional \textit{$\field$-homology sphere} if $\tilde{H}_i(\lk_\Delta \sigma;\field)\cong \tilde{H}_i(\mathbb{S}^{d-1-|\sigma|};\field)$ for every face $\sigma\in\Delta$ (including the empty face) and $i\geq -1$. (Here we denote by $\tilde{H}_*(\Delta,\field)$ the reduced homology with coefficients in a field $\field$.) Similarly, $\Delta$ is a $(d-1)$-dimensional $\field$-\emph{homology manifold} if all of its vertex links are $(d-2)$-dimensional $\field$-homology spheres. A $(d-1)$-dimensional simplicial complex $\Delta$ is called a \emph{normal $(d-1)$-pseudomanifold} if (i) it is pure and connected, (ii) every $(d-2)$-face of $\Delta$ is contained in exactly two facets and (iii) the link of each face of dimension $\leq d-3$ is also connected. For a fixed $d$, we have the following hierarchy:
		\begin{center}
			polytopal $(d-1)$-spheres $\subseteq$ homology $(d-1)$-spheres $\subseteq$ connected homology $(d-1)$-manifolds $\subseteq$ normal $(d-1)$-pseudomanifolds.
		\end{center}
		When $d=3$, the first two classes and the last two classes of complexes above coincide; however, starting from $d=4$, all of the inclusions above are strict.

		For a $(d-1)$-dimensional simplicial complex $\Delta$, the \textit{$f$-number} $f_i = f_i(\Delta)$ denotes the number of $i$-dimensional faces of $\Delta$. The vector $(f_{-1}, f_0, \cdots, f_{d-1})$ is called the $f$-\textit{vector} of $\Delta$. We also define the \textit{$h$-vector} $h(\Delta)=(h_0,\cdots, h_d)$ by the relation $\sum_{j=0}^{d}h_j\lambda^{d-j}=\sum_{i=0}^{d}f_{i-1}(\lambda-1)^{d-i}$. If $\Delta$ is a homology $(d-1)$-sphere, then by the Dehn-Sommerville relations, $h_i(\Delta)=h_{d-i}(\Delta)$ for all $0\leq i\leq d$. Hence it is natural to consider the successive differences between the $h$-numbers: we form a vector called the \emph{$g$-vector}, whose entries are given by $g_0=1$ and $g_i=h_i-h_{i-1}$ for $1\leq i\leq \left\lfloor d/2\right\rfloor$. The $f$-vector and $h$-vector of any homology sphere are determined by its $g$-vector. The following lemma, which was first stated by McMullen \cite{M} for shellable complexes and later generalized to all pure complexes by Swartz \cite[Proposition 2.3]{S4}, is a useful fact for face enumeration.
		  \begin{lemma}\label{lm: Swartz}
		  	If $\Delta$ is a pure $(d-1)$-dimensional simplicial complex, then for $k\geq 1$,
		  	\[\sum_{v\in V(\Delta)} g_k(\lk_\Delta v)=(k+1)g_{k+1}(\Delta)+(d+1-k)g_k(\Delta).\]
		  \end{lemma}
	
	\subsection{The generalized lower bound theorem for polytopes}
	Theorem \ref{thm: Kalai} provides a full description of normal pseudomanifolds with $g_2=0$. To characterize the simplicial polytopes with $g_i=0$ for $i\geq 3$, we need to generalize stackedness. Following Murai and Nevo \cite{MN}, given a simplicial complex $\Delta$ and $i\geq 1$,  we let 
	\[\Delta(i):=\{\sigma\subseteq V(\Delta)\mid\skel_i(2^\sigma)\subseteq \Delta\},\]
	where $2^\sigma$ is the power set of $\sigma$. (In other words, we add to $\Delta$ all simplices whose $i$-dimensional skeleton is contained in $\Delta$.) 
	
	A \emph{homology $d$-ball} (over a field $\field$) is a $d$-dimensional simplicial complex $\Delta$ such that (i) $\Delta$ has the same homology as the $d$-dimensional ball, (ii) for every face $F$, the link of $F$ has the same homology as the $(d-|F|)$-dimensional ball or sphere, and (iii) the boundary complex, $\partial\Delta:=\{F\in \Delta\mid \tilde{H}_i(\lk_\Delta F)=0, \forall i\}$, is a homology $(d-1)$-sphere. If $\Delta$ is a homology $d$-ball, the faces of $\Delta-\partial \Delta$ are called the \emph{interior} faces of $\Delta$. If furthermore $\Delta$ has no interior $k$-faces for $k\leq d-r$, then $\Delta$ is said to be $(r-1)$-\emph{stacked}. An $(r-1)$-\emph{stacked} homology sphere (resp. simplicial sphere) is the boundary complex of an $(r-1)$-stacked triangulation of a homology ball (resp. simplicial ball). It is easy to see that being stacked is equivalent to being 1-stacked. The following theorem is a part of the generalized lower bound theorem established by Murai and Nevo, see \cite[Theorem 1.2 and Lemma 2.1]{MN}. 
	\begin{theorem}\label{thm: GLBT}
		Let $\Delta$ be a polytopal $(d-1)$-sphere and $2\leq r\leq d/2$. Then $g_r(\Delta)=0$ if and only if $\Delta$ is $(r-1)$-stacked. Furthermore, if that happens, then $\Delta(d-r)=\Delta(r-1)$ is a simplicial $d$-ball.
	\end{theorem}
	\subsection{Rigidity Theory}
	We give a short presentation of rigidity theory that will be used in later sections. Let $G=(V,E)$ be a graph. A \emph{$d$-embedding} is a map $\phi: V\to \R^d$. It is called \emph{rigid} if there exists an $\epsilon>0$ such that if $\psi: V\to \R^d$ satisfies $\dist(\phi(u), \psi(u))<\epsilon$ for every $u\in V$ and $\dist(\psi(u), \psi(v))=\dist(\phi(u), \phi(v))$ for every $\{u,v\}\in E$, then $\dist(\psi(u), \psi(v))=\dist(\phi(u), \phi(v))$ for every $u,v\in V$. (Here $\dist$ denotes the Euclidean distance.) A graph $G$ is called \emph{generically $d$-rigid} if the set of rigid $d$-embeddings of $G$ is open and dense in the set of all $d$-embeddings of $G$. 
	
Given a graph $G$ and a $d$-embedding $\phi$ of $G$, we define the matrix $\mathrm{Rig}(G, \phi)$ associated with a graph $G$ as follows: it is an $f_1(G) \times df_0(G)$ matrix with rows labeled by edges of $G$ and columns grouped in blocks of size $d$, with each block labeled by a vertex of $G$; the row corresponding to $\{u,v\}\in E$ contains the vector $\phi(u)-\phi(v)$ in the block of columns corresponding to $u$, the vector $\phi(v)-\phi(u)$ in columns corresponding to $v$, and zeros everywhere else. It is easy to see that for a generic $\phi$ the dimensions of the kernel and image of $\mathrm{Rig}(G, \phi)$ are independent of $\phi$. Hence we define the \emph{rigidity matrix} of $G$ as $\mathrm{Rig}(G, d)=\mathrm{Rig}(G, \phi)$ for a generic $\phi$. Given a $d$-embedding $f: V\to \R^d$, a stress with respect to $f$ is a function $w:E\to \R$ such that for every vertex $v\in V$ 
	\[ \sum_{u:\{v,u\}\in E} w(\{v,u\})(f(v)-f(u))=0.\]
	We say that an edge $\{u,v\}$ participates in a stress $w$ if $w(\{u,v\})\neq 0$, and that a vertex $v$ participates in $w$ if there is an edge that participates in $w$ and contains $v$. The following three lemmas summarize a few basic results of rigidity theory. For a simplicial complex $\Delta$, we denote the graph of $\Delta$ (equivalently, the 1-skeleton of $\Delta$) by $G(\Delta)$. We say a simplicial complex $\Delta$ is generically $d$-rigid if $G(\Delta)$ is generically $d$-rigid.
	\begin{lemma}\label{lm: basic facts }
		Let $\Delta$ be a simplicial complex.
		\begin{enumerate}\label{lm: basic rigidity}
			\item \rm{(Cone lemma, \cite{Whiteley}, \cite{K} and \cite{T})} For an arbitrary $v\in V(\Delta)$ and any $d$, $\lk_\Delta v$ is generically $(d-1)$-rigid if and only if $\st_\Delta v$ is generically $d$-rigid. 
			\item \rm{(\cite{F})} If $\Delta$ is a normal $(d-1)$-pseudomanifold, then $\Delta$ is generically $d$-rigid.
			\item If $\Delta$ is generically $d$-rigid, then $g_2(\Delta)=\dim {\rm{LKer}}(\mathrm{Rig}(\Delta,d))$, where ${\rm{LKer}}(M)$ is the left null space of a matrix $M$.			
		\end{enumerate}
	\end{lemma}
	The next lemma was originally stated in \cite{NN} for the class of homology spheres. Since the proof given in \cite{NN} only uses the fact that vertex links of these complexes are generically $(d-1)$-rigid and that the facet-ridge graph of the antistar of any vertex is connected, part 2 of Lemma \ref{lm: basic rigidity} allows us to generalize the statement to the class of normal pseudomanifolds. (For details about facet-ridge graphs of normal pseudomanifolds and their connectivity, see, for instance, \cite[Section 2]{BD}.)
	\begin{lemma}\label{lm: NN, vertex stress}
	 	\rm{(\cite[Proposition 2.10]{NN})} Let $d\geq 4$ and let $\Delta$ be a prime normal $(d-1)$-pseudomanifold. Then every vertex $u\in \Delta$ participates in a generic $d$-stress of the graph of $\Delta$.
	\end{lemma}
	The following result is proved in Kalai's paper \cite[Theorem 7.3]{K}.
	\begin{lemma}\label{lm: coning stress}
		Let $d\geq 4$. For any generically $d$-rigid pure $(d-1)$-dimensional simplicial complex $\Delta$, $g_2(\lk_\Delta v)\leq  g_2(\Delta)$.
	\end{lemma}
	\section{Retriangulations of simplicial complexes}
	A \emph{triangulation} of a topological space $M$ is any simplicial complex
	$\Delta$ such that the geometric realization of $\Delta$ is homeomorphic to $M$. In this section, we introduce three operations that produce new triangulations of the original topological space. We will use these operations extensively to characterize homology manifolds with $g_2\leq 2$. The first one is called the central retriangulation, see \cite[Section 5]{S2}.
    \begin{definition}
    	Let $\Delta$ be a $d$-dimensional simplicial complex and $B$ be a subcomplex of $\Delta$; assume also that $B$ is a simplicial $d$-ball. The \emph{central retriangulation} of $\Delta$ along $B$, denoted as $\crtr_B(\Delta)$, is the new complex we obtain after removing all of the interior faces of $B$ and replacing them with the interior faces of the cone on the boundary of $B$, where the cone point is a new vertex $u$. 
    \end{definition}
     \begin{figure}[h]
      	\centering
      	\includegraphics{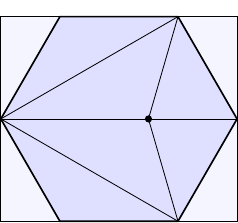}
      	\hspace{25mm}
      	\includegraphics{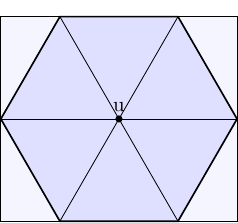}
      	\caption{Central retriangulation along a subcomplex $B$ (the darker blue region), where $B$ has six interior edges, one interior vertex, and $\partial B$ is a 6-cycle.}
      \end{figure}
     Recall that the \emph{stellar subdivision} of a simplicial complex $\Delta$ at the face $\tau$ is \[\sd_\tau(\Delta)=(\Delta\backslash\tau)\cup(u*\partial(\st_\Delta \tau)),\] where $u$ is the newly added vertex. It immediately follows from the definition that $\crtr_{\st_\Delta \tau}(\Delta)=\sd_\tau(\Delta)$. In this paper, we will mainly discuss central retriangulations of $\Delta$ along an $(r-1)$-stacked ($2\leq r\leq d/2$) subcomplex. The following lemma indicates how the $g$-vector changes under central retriangulations.
    \begin{lemma}\label{lm: g_k central retriangulation}
    	Let $\Delta$ be a $d$-dimensional simplicial complex and $B\subseteq\Delta$ be an $(r-1)$-stacked $d$-dimensional ball, where $2\leq r\leq d/2$. Then $g_i(\crtr_B(\Delta))=g_i(\Delta)+g_{i-1}(\partial B)$ for $1\leq i\leq d/2$.
    \end{lemma}
    \begin{proof}
    	Since $B$ is $(r-1)$-stacked, $B$ has no interior faces of dimension $\leq  d-r$. Hence by the definition of central retriangulation, $f_i(\crtr_B(\Delta))=f_i(\Delta)+f_{i-1}(\partial B)$ for $0\leq i\leq d-r$. Now use the formula $g_j(\Gamma)=\sum_{i=0}^{j}(-1)^{j-i}\binom{d+2-i}{j-i}f_{i-1}(\Gamma)$, where $d=\dim\Gamma+1$, together with an observation that $\dim \partial B=\dim \Delta-1$ to obtain $g_i(\crtr_B(\Delta))=g_i(\Delta)+g_{i-1}(\partial B)$. 
    \end{proof}
    
    If $P$ is a $d$-polytope, $H$ a supporting hyperplane of $P$ such that $H^+$ is the closed half-space determined by $H$ that contains $P$, and $v\in \mathbb{R}^d\backslash H$, then we say that $v$ is \emph{beneath} $H$ (with respect to $P$) if $v\in H^+$ and $v$ is \emph{beyond} $H$ if $v\notin H^+$. In the following lemma we denote the set of missing $k$-faces of $\Delta$ by $M_k(\Delta)$.
    \begin{lemma}\label{lm: other prop of central retriangulation}
    	Let $\Delta$ be a homology $d$-manifold and $\tau$ be an $i$-face of $\Delta$. Then the following holds:
    	\begin{enumerate}
    	\item If $i>d/2$, then $\st_\Delta \tau$ is a $(d-i)$-stacked homology ball.
    	\item $M_k(\sd_\tau(\Delta))=\big(M_k(\Delta)-\{F\in  M_k(\Delta)\mid\tau\subseteq F\}\big)\cup \{u*F\mid F\in\Delta,\; F\in M_{k-1}(\st_\Delta \tau)\}$. Here $u$ is the new vertex of the retriangulation.
    	\item If $\Delta$ is a polytopal $d$-sphere, then $\sd_\tau(\Delta)$ is also a polytopal $d$-sphere. 
    	\end{enumerate}
    \end{lemma}
    \begin{proof}
    	Part 1 and 2 follow from the definitions. For part 3, we let $P$ be a $d$-polytope whose boundary complex coincides with $\Delta$ and we let $H_F$ be the supporting hyperplane of a facet $F$. There exists a point $p\in \mathbb{R}^d$ such that $p$ is beyond all hyperplanes $H_F$ for $\tau\subseteq F$, and beneath all $H_F$ for $\tau\notin F$. Then by \cite[Theorem 1 in Section 5.2]{G}, $\sd_\tau(\Delta)$ is the boundary complex of $\conv(V(\Delta)\cup\{p\})$, which is a polytope.
    \end{proof}
    
    Next we introduce the second retriangulation, which in a certain sense is the inverse of central retriangulation along an $(r-1)$-stacked subcomplex.
    \begin{definition}
    	Let $\Delta$ be a $d$-dimensional simplicial complex. Assume that there is a vertex $v\in V(\Delta)$ such that $\lk_\Delta v$ is an $(r-1)$-stacked homology $(d-1)$-sphere. If no interior face of $(\lk_\Delta v)(r-1)$ is a face of $\Delta$, then define the \emph{inverse stellar retriangulation} on vertex $v$ by
    	\[\sd_v^{-1}(\Delta)=(\Delta\backslash\{v\})\cup(\lk_\Delta v)(r-1).\]In other words, we replace the star of $v$ with the ball $(\lk_\Delta v)(r-1)$. It is easy to see that $\sd_v^{-1}(\Delta)$ is PL-homeomorphic to $\Delta$. Using the same argument as in Lemma \ref{lm: g_k central retriangulation}, we prove the following result.
    \end{definition}
    
    \begin{lemma}\label{lm: g_k inverse retriangulation}
    	Let $\Delta$ be a $d$-dimensional simplicial complex. If $\lk_\Delta v$ is an $(r-1)$-stacked homology $(d-1)$-sphere for some $2\leq r\leq \frac{d+1}{2}$ and no interior face of $(\lk_\Delta v)(r-1)$ is a face of $\Delta$, then $\sd_v^{-1}(\Delta)$ is well-defined and $g_i(\sd_v^{-1}(\Delta))=g_i(\Delta)-g_{i-1}(\lk_\Delta v)$ for $1\leq i\leq d/2$.
    \end{lemma}
    \begin{figure}[h]
    \centering
    \subfloat[A stacked vertex link, $\lk_\Delta v$, in a 3-dimensional complex $\Delta$]{\includegraphics[scale=1.4]{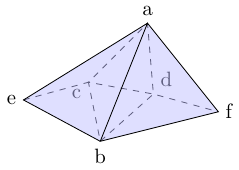}}
    \hspace{15mm}
    \subfloat[Two missing faces $\{a,b,c\}$ and $\{a,b,d\}$ in $\lk_\Delta v$, and three missing facets in $(\lk_\Delta v)\cup\{ \{a,b,c\} \cup \{a,b,d\}\}$]{\includegraphics[scale=1.4]{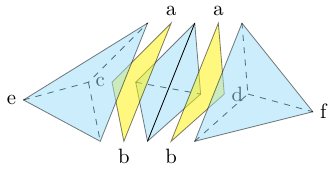}}
    \caption{Constructing $\sd_v^{-1}(\Delta)$ from $\Delta$: remove the vertex $v$ and add all five missing faces above to $\Delta$.}
    \end{figure}
    A similar retriangulation that reduces $g_2$ was introduced by Swartz \cite{S}. In contrast with the inverse stellar retriangulation, the number of vertices, or equivalently $g_1$, is not necessarily reduced in Swartz's operation.
    \begin{definition}
    	Let $\Delta$ be a $d$-dimensional simplicial complex such that one of the vertex links, $\lk_\Delta v$, is a homology $(d-1)$-sphere. If a missing facet $\tau$ of $\lk_\Delta v$ is also a missing face of $\Delta$, then we define the \textit{Swartz operation} on $(v,\tau)$ of $\Delta$ by first removing $v$, next adding $\tau$, then coning off two remaining homology spheres $S_1, S_2$ with two new vertices $v_1,v_2$. (Here $S_1$, $S_2$ are the two homology spheres such that their connected sum by identifying the face $\tau$ is $\lk_\Delta v$.) If one of the two spheres, say $S_1$, forms the boundary of a missing facet of $\Delta\cup\{\tau\}$, then we simply add this missing facet to $\Delta\cup\{\tau\}$ instead of coning off $S_1$ with $v_1$. The resulting complex is denoted by $\so_{v,\tau}(\Delta)$. If the dimension of $\Delta$ is at least three, then iterating this process, we are able to add all missing facets of $\lk_\Delta v$ to $\Delta$. The resulting complex is denoted by $\so_{v}(\Delta)$. 
    \end{definition}
    \begin{figure}[h]
    	\centering
    	 \subfloat[A 2-sphere $\Delta$: $\lk_\Delta v$ is a 6-cycle and $\tau=\{a,b\}$ is a missing edge of $\Delta$.]{\includegraphics[scale=1.215]{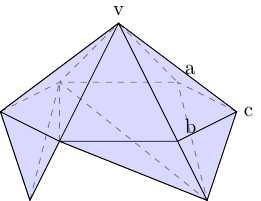}}
    	 \hspace{25mm}
    	 \subfloat[A retriangulation $\Delta'=\so_{v,\tau}(\Delta)$: add $\{a,b\}$ to $\Delta'$ and replace $\st_\Delta v$ with  $\st_{\Delta'} v'\cup \{a,b,c\}$]{\includegraphics[scale=1.215]{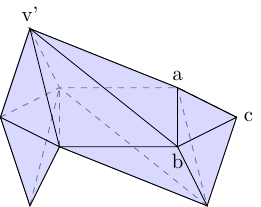}}
    	 \caption{The Swartz operation on a 2-sphere}
    \end{figure}
    Note that $\so_{v}(\Delta)$ is indeed well-defined since the construction is independent of the order of missing facets of $\lk_\Delta v$ chosen. Also $\so_v(\Delta)$ is PL-homeomorphic to $\Delta$, and if $\lk_\Delta v$ is a stacked sphere of dimension $\geq 3$, then $\so_v(\Delta)=\sd_v^{-1}(\Delta)$. If $\Delta$ is of dimension $\geq 3$, then since $g_2(\so_{v,\tau}(\Delta))=g_2(\Delta)-1$ by \cite{S}, we obtain that \[g_2(\so_v(\Delta))=g_2(\Delta)-\#\{\mathrm{missing}\; \mathrm{facets}\;\mathrm{of}\;\lk_\Delta v\}.\]
    \begin{lemma}\label{lm: g_k Swartz's operation}
    	Let $\Delta$ be a normal $(d-1)$-pseudomanifold for $d\geq 4$. If a vertex link, $\lk_\Delta v$, is a homology $(d-2)$-sphere and there are $k$ missing facets of $\lk_\Delta v$ that are not faces of $\Delta$, then $g_2(\Delta)\geq k$.
    \end{lemma}
    \section{From $g_2=0$ to $g_2=1$}
    The goal of this section is to provide an alternative proof of Theorem \ref{thm: NevoNovinsky} for the case of $d\geq 5$ but in a much larger class -- that of normal pseudomanifolds, see Theorem \ref{thm: g_2=1 normal pseudomanifold}. 
    
    If $\Delta$ is a stacked $(d-1)$-sphere, and $\tau$ is a face of $\Delta$ with the property that $\lk_\Delta \tau$ is the boundary complex of a simplex, then $\partial (\st_\Delta \tau)=\partial\tau*\lk_\Delta \tau$ is a join of two boundary complexes of simplices, and hence has $g_1=1$. Therefore by Lemma \ref{lm: g_k central retriangulation}, centrally retriangulating $\Delta$ along $\st_\Delta \tau$ results in a $(d-1)$-sphere with $g_2=1$. However, the resulting complex is not necessarily prime. In the rest of the paper, we denote by $\mathcal G_d$ the set of complexes that is either the join of $\partial \sigma^i$ and $\partial \sigma^{d-i}$, where $2\leq i\leq d-2$, or the join of $\partial \sigma^{d-2}$ and a cycle. The following lemma is a special case of Theorem 1(a) in \cite{BD}. We give a proof for the sake of completeness.
   \begin{lemma}\label{lm: f_0=d+2}
   	If $\Delta$ is a normal $(d-1)$-pseudomanifold on $d+2$ vertices, then $\Delta$ is the join of two boundary complexes of simplices.
   \end{lemma}
   \begin{proof}
   	Let $\sigma$ be a missing $i$-face of $\Delta$, $1\leq i\leq d-1$. Then for any $(i-1)$-face $\tau\subseteq\partial\sigma$, $\lk_\Delta \tau$ is a normal $(d-i-1)$-pseudomanifold, and $d-i+1\leq f_0(\lk_\Delta \tau)\leq f_0(\Delta)-f_0(\sigma)=d-i+1$. Hence $\lk_\Delta \tau$ is the boundary complex of a $(d-i)$-simplex. Furthermore, $V(\sigma)\sqcup V(\lk_\Delta \tau)=V(\Delta)$ and the link of every $(i-1)$-face of $\sigma$ must be exactly $\lk_\Delta \tau$. This implies $\Delta\supseteq \partial\sigma*\lk_\Delta \tau$. Since $\partial\sigma*\lk_\Delta \tau$ is a normal $(d-1)$-pseudomanifold and no proper subcomplex of $\Delta$ can be a normal $(d-1)$-pseudomanifold, it follows that $\Delta=\partial\sigma*\lk_\Delta \tau$.
   \end{proof}
    \begin{proposition}\label{prop: recover prime g_2=1 cases}
    	Let $\Delta$ be a stacked $(d-1)$-sphere and let $\tau$ be a ridge. If $\sd_\tau(\Delta)$ is prime, then $\sd_\tau(\Delta)$ is the join of $\partial \sigma^{d-2}$ and a cycle. In particular, $\sd_\tau(\Delta)\in \mathcal{G}_d$. 
    \end{proposition}
    \begin{proof}
    	Assume that $\Delta=\Delta_1\#\Delta_2\#\cdots \#\Delta_{n+1}$, where $\Delta_1,\cdots, \Delta_{n+1}$ are boundary complexes of $d$-simplices. Assume further that $\tau_1,\cdots, \tau_n$ are the missing facets of $\Delta$. Since $\sd_\tau(\Delta)$ is prime, by part 2 of Lemma \ref{lm: other prop of central retriangulation}, every missing facet of $\Delta$ must contain $\tau$. So there exist distinct vertices $v_0,\cdots,v_{n+1}$ of $V(\Delta)$ such that $\tau_i=\tau\cup\{v_i\}$, $\Delta_1=\partial(v_0*\tau_1)$ and $\Delta_{n+1}=\partial(v_{n+1}*\tau_{n})$. It follows that $\Delta=\partial(\tau*P)$, where $P$ is the path $(v_0,v_1,\cdots,v_{n+1})$. Hence $\sd_\tau(\Delta)=\partial \tau *\tilde{P}\in \mathcal{G}_d$, where $\tilde{P}$ is the (graph) cycle obtained by adding the new vertex in $\sd_\tau(\Delta)$ and connecting it to the endpoints of $P$.				
        \end{proof}
        
        The next lemma (\cite[Corollary 1.8]{S3}) places restrictions on the first few $g$-numbers for normal pseudomanifolds. (See page 56 in \cite{St2} for definition of $g_i^{<k>}$ and $M$-vector.)
        \begin{lemma}\label{lm: g-vector of normal pseudo}
        	Let $\Delta$ be a normal $(d-1)$-pseudomanifold with $d\geq 4$. Then $g_3\leq g_2^{<2>}$. In particular, if $g_3\geq 0$, then $(1,g_1,g_2,g_3)$ is an $M$-vector.
        \end{lemma}
    		
    	Now we are ready to give an alternative proof of Theorem \ref{thm: NevoNovinsky} for dimension $d-1\geq 4$.
    	\begin{theorem}\label{thm: g_2=1 normal pseudomanifold}
    		Let $\Delta$ be a prime normal $(d-1)$-pseudomanifold with $g_2(\Delta)=1$ and $d\geq 5$. Then $\Delta$ is the stellar subdivision of a stacked $(d-1)$-sphere at a face of dimension $i$, where $0<i<d-1$. Furthermore, $\Delta\in \mathcal{G}_d$.
    	\end{theorem}
    	\begin{proof}
    	By Lemma \ref{lm: g-vector of normal pseudo}, $g_3(\Delta)\leq g_2(\Delta)^{<2>}=1$. Also by Lemma \ref{lm: basic rigidity}, since $\lk_\Delta v$ and $\Delta$ are normal pseudomanifolds of dimension $d-2$ and $d-1$ respectively, they are generically $(d-1)$- and $d$-rigid respectively. Hence by Lemma \ref{lm: coning stress}, $0\leq g_2(\lk_\Delta v)\leq g_2(\Delta)\leq 1$. Using Lemma \ref{lm: Swartz}, we obtain \[\sum_{v\in V(\Delta)} g_2(\lk_\Delta v)=d-1+3g_3(\Delta)\leq d+2.\]
    									
    	If $g_2(\lk_\Delta v)=1$ for every vertex $v$ of $\Delta$, then the above inequality implies that $f_0(\Delta)\leq d+2$. However, $f_0(\Delta)\geq d+2$ and so $\Delta$ has exactly $d+2$ vertices. Hence by Lemma \ref{lm: f_0=d+2}, $\Delta=\partial \sigma^i *\partial \sigma^{d-i}$ for some $2\leq i\leq d-2$, i.e., $\Delta\in \mathcal{G}_d$. It is easy to see that in this case $\Delta$ is the stellar subdivision of $\partial \sigma^d$ at an $i$-face.
    									
        Otherwise, there exists a vertex $v$ such that $g_2(\lk_\Delta v)=0$. By Theorem \ref{thm: Kalai}, $\lk_\Delta v$ is a stacked sphere. We claim that every missing facet $\tau$ of $\lk_\Delta v$ is not a face of $\Delta$; otherwise, $\tau\in\Delta$ and $v*\tau$ is a missing facet of $\Delta$, contradicting the fact that $\Delta$ is prime. Hence we may apply the inverse stellar retriangulation on the vertex $v$ to obtain a new normal pseudomanifold $\sd_v^{-1}(\Delta)$. By Theorem \ref{thm: Kalai} and Lemma \ref{lm: g_k inverse retriangulation}, $0\leq g_2(\sd_v^{-1}(\Delta))=g_2(\Delta)-g_1(\lk_\Delta v)\leq 0$, which implies that $\sd_v^{-1}(\Delta)$ is a stacked sphere. Furthermore, $g_1(\lk_\Delta v)=1$, so $\lk_\Delta v$ is the connected sum of two boundary complexes of simplices. This implies that $\Delta$ is the stellar subdivision of $\sd_v^{-1}(\Delta)$ at a ridge (the unique missing facet of $\lk_\Delta v$). This proves the first claim. Finally, the second claim follows immediately from Proposition \ref{prop: recover prime g_2=1 cases}.
    	\end{proof}
    \section{From $g_2=1$ to $g_2=2$}	
    In this section, we find all homology $(d-1)$-manifolds with $g_2=2$ for $d\geq 4$. Our strategy, as in the previous section, is to apply certain central retriangulations to homology $(d-1)$-spheres with $g_2=1$ and show that in this way we obtain all homology manifolds with $g_2=2$, apart from one exception in dimension 3. We begin with a few lemmas.
      \begin{lemma}\label{lm: prime impliex vertex link is prime}
      	Let $d\geq 5$ and let $\Delta$ be a prime normal $(d-1)$-pseudomanifold with $g_2(\Delta)=2$. Furthermore, assume that $g_2(\lk_\Delta v)\geq 1$ for every vertex $v\in V(\Delta)$. Then every vertex link of $\Delta$ with $g_2=1$ is prime.
      \end{lemma}
      \begin{proof}
      	Assume by contradiction that $g_2(\lk_\Delta u)=1$ and $\lk_\Delta u$ is not prime for some vertex $u\in V(\Delta)$. Then by Theorem \ref{thm: NevoNovinsky}, $\lk_\Delta u$ can be written as $\Delta_1\#\Delta_2\#\cdots \#\Delta _k$, where $k\geq 2$, $\Delta_1\in \mathcal{G}_{d-1}$ and the other $\Delta_i$'s are boundary complexes of simplices. First we claim that every missing facet $\tau$ of $\lk_\Delta u$ is not a face of $\Delta$. Otherwise, $\tau* u$ is a missing facet of $\Delta$, contradicting that $\Delta$ is prime. Applying the Swartz operation on vertex $u$ (with a new vertex $u'$), we obtain a new normal $(d-1)$-pseudomanifold $\Delta':=\so_u(\Delta)$ and $g_2(\Delta')=g_2(\Delta)-(k-1)=3-k$. Since $g_2(\Delta')\geq 0$, it follows that $k\leq 3$. 
      	
      	Since $\st_{\Delta'}u'$ is generically $d$-rigid and $g_2(\st_{\Delta'}u')=g_2(\lk_{\Delta'}u')=1$, there is a nontrivial stress of $\Delta'$ supported on $\st_{\Delta'}u'$, and so $k\neq 3$. Next if $k= 2$, then the link of the vertex $w=V(\Delta_2\backslash \Delta_1)$ has $g_2(\lk_{\Delta'} w)=g_2(\lk_{\Delta}w)\geq 1$. Hence there exists a generic stress of $\Delta'$ supported on $\st_{\Delta'}w$, and $w$ participates in this stress. Since $w\notin\st_{\Delta'}u'$, we must have $g_2(\Delta')\geq 2$, contradicting the fact that $g_2(\Delta')=1$. We conclude that $k=1$ and $\lk_\Delta u$ is prime.
      \end{proof}
      
    \begin{lemma}\label{lm: g_2=2 and links}
     	Let $d\geq 5$ and let $\Delta$ be a prime normal $(d-1)$-pseudomanifold with $g_2(\Delta)=2$. Furthermore, assume that $g_2(\lk_\Delta v)\geq 1$ for every vertex $v\in V(\Delta)$. Then the following holds:
     	\begin{enumerate}
     		\item If $g_2(\lk_\Delta u)=2$ for some vertex $u$, then $V(\st_\Delta u)=V(\Delta)$.
     		\item If $g_2(\lk_\Delta u)=1$ and $G(\Delta[V(\lk_\Delta u)])=G(\lk_\Delta u)\cup \{e\}$ for some vertex $u$ and edge $e$, then $\Delta=\partial \sigma^1*\partial \sigma^2*\partial \sigma^{d-3}$.
     		\item If every vertex $u$ with $g_2(\lk_\Delta u)=1$ also satisfies $G(\Delta[V(\lk_\Delta u)])=G(\lk_\Delta u)$, then at least one of such vertex links is the join of two boundary complexes of simplices. 
     	\end{enumerate}
     \end{lemma}
     \begin{proof}
     	For part 1, note that $g_2(\lk_\Delta u)=g_2(\st_\Delta u)=2$. If $V(\st_\Delta u)\neq V(\Delta)$, then by Lemma \ref{lm: NN, vertex stress}, there is a vertex not in $V(\st_\Delta u)$ that participates in a generic $d$-stress of $G(\Delta)$. Hence $g_2(\Delta)\geq g_2(\st_\Delta u)+1=3$, contradicting $g_2(\Delta)=2$.
     	
     	For part 2, note that $g_2(\Delta[V(\st_\Delta u)])= g_2(\st_\Delta u)+1=2$, and so using the same argument as in part 1 we obtain that $V(\st_\Delta u)=V(\Delta)$. Since $G(\lk_\Delta u)$ is not a complete graph (it misses $e$) and $\lk_\Delta u$ is prime by Lemma \ref{lm: prime impliex vertex link is prime}, it follows from Theorem \ref{thm: NevoNovinsky} that $\lk_\Delta u$ is the join of a cycle $C$ and $\partial \sigma^{d-3}$. Hence $V(e)\subseteq V(C)$. For every vertex $v\in V(C)-V(e)$, its degree in $\Delta[V(C)]=C\cup\{e\}$ is exactly 2, and thus $V(\st_\Delta v)\subsetneq V(\Delta)$. By part 1, $g_2(\lk_\Delta v)=1$. Then as $V(\st_\Delta v)$ is strictly contained in $V(\Delta)$, $f_0(\lk_\Delta v)\leq 3+f_0(\partial \sigma^{d-3})=d+1$ yields that $\lk_\Delta v$ is the join of a 3-cycle and $\partial \sigma^{d-3}$, which further implies that $e\in \lk_\Delta v$. Hence $\Delta[V(C)]$ is a triangulated 2-ball whose boundary is the 4-cycle $C$. Since $V(\st_\Delta u)=V(\Delta)$, $\Delta[V(\lk_\Delta u)]$ is a homology ball. Hence it follows that $\Delta[V(\lk_\Delta u)]=\Delta[V(C)]*\partial \sigma^{d-3}$, and so $\Delta$ is the suspension of $\partial (e*u)*\partial \sigma^{d-3}$.
     	
     	For part 3, by Lemma \ref{lm: prime impliex vertex link is prime} and Theorem \ref{thm: NevoNovinsky}, either $\lk_\Delta u=C *\partial F$ for some cycle $C$ of length $\geq 4$ and missing $(d-3)$-face $F$ of $\lk_\Delta u$, or $\lk_\Delta u=\partial \sigma^i*\partial \sigma^{d-1-i}$ for some $2\leq i\leq d-3$. If every vertex link with $g_2=1$ is of the former type, then since $G(\Delta[V(C)])=G(C)$, it follows that $V(\st_\Delta a)\subsetneq V(\Delta)$ for every vertex $a\in V(C)$. Hence by part 1, $g_2(\lk_\Delta a)=1$, and it is the join of $\partial F$ and a cycle. Also every vertex from $\Delta-\st_\Delta u$ is not connected to $u$, so again by part 1 the links of these vertices have $g_2=1$. On the other hand, the link of every vertex $b\in F$ contains a subcomplex $u*\partial(F-\{b\})*C$ yet $(F-\{b\})*C\nsubseteq \lk_\Delta b$. Hence $g_2(\lk_\Delta b)\neq 1$ by Theorem \ref{thm: NevoNovinsky}. By Lemma \ref{lm: coning stress}, $g_2(\lk_\Delta b)=2$, and by part 1, $V(\st_\Delta b)=V(\Delta)$ for every $b\in F$. 
     	
     	We claim that $F$ is a missing face of $\Delta$. Otherwise, let w be a vertex in $\lk_\Delta F$. Since $g_2(\lk_\Delta w)=1$, the previous argument shows that $\lk_\Delta w$ must be the join of $\partial F$ and a cycle. However, $F\in \lk_\Delta w$, a contradiction. Now $\lk_\Delta u=\Delta[V(\lk_\Delta u)]$, so we apply the inverse stellar retriangulation on vertex $u$ to obtain a new complex $\sd_{u}^{-1}(\Delta)$ with $0\leq g_2(\sd_{u}^{-1}(\Delta))=g_2(\Delta)-g_1(\lk_\Delta u)\leq 2-2=0$. Hence $\sd_{u}^{-1}(\Delta)$ is stacked. On the other hand, there exists a vertex $z$ in $V(\sd_{u}^{-1}(\Delta))-V(\lk_\Delta u)$ and its link in $\sd_{u}^{-1}(\Delta)$ is also stacked. But then $g_2(\lk_\Delta z)=g_2(\lk_{\sd_{u}^{-1}(\Delta)} z)=0$, which contradicts our assumption $g_2(\lk_\Delta z)\geq 1$. The result follows.
     \end{proof}
   
    \begin{theorem}\label{thm:g_2=2 and d>4}
    	Let $d\geq 5$. Every prime normal $(d-1)$-pseudomanifold with $g_2=2$ can be obtained from a polytopal $(d-1)$-sphere with $g_2=0$ or 1, by centrally retriangulating along some stacked subcomplex.
    \end{theorem}
    \begin{proof}
    	Let $\Delta$ be the normal $(d-1)$-pseudomanifold with $g_2(\Delta)=2$. By Lemma \ref{lm: g-vector of normal pseudo}, $g_3(\Delta)\leq g_2^{<2>}(\Delta)=2$. Also by Lemma \ref{lm: Swartz}, \[\sum_{v\in V(\Delta)}g_2(\lk_\Delta v)=2(d-1)+3g_3(\Delta)\leq 2d+4.\]
    	In the following we consider two different cases.
    	
    	{\bf Case 1:} $g_2(\lk_\Delta v)\geq 1$ for every vertex $v\in V(\Delta)$. First notice that there exists a vertex $u\in V(\Delta)$ with $g_2(\lk_\Delta u)=1$. Otherwise, $2f_0(\Delta)=\sum_{v\in V(\Delta)}g_2(\lk_\Delta
     v)\leq 2d+4$, and by Lemma \ref{lm: f_0=d+2} $g_2(\Delta)\leq 1$, a contradiction. Then Lemma \ref{lm: prime impliex vertex link is prime} and Lemma \ref{lm: g_2=2 and links} imply that either $\Delta=\partial \sigma^1* \partial \sigma^2 *\partial \sigma^{d-3}$, or there exist a vertex $u$ such that $\lk_\Delta u=\partial \sigma^i *\partial \sigma^{d-i-1}$ for some $i$. In the former case, $\Delta$ is exactly the complex $\crtr_{\st_\Delta \tau}(\partial \sigma^2*\partial \sigma^{d-2})$, where $\tau$ is a facet of $\partial \sigma^{d-2}$.
     Now we deal with the latter case by first determining the $g_2$-numbers of all vertex links. If $w\in V(\lk_\Delta u)$, then $\lk_\Delta w$ contains either the subcomplex $u*\partial \sigma^i *\partial \sigma^{d-i-2}$ or $u*\partial \sigma^{i-1} *\partial \sigma^{d-i-1}$. Hence $\lk_\Delta w\notin \mathcal{G}_{d-1}$ and we conclude that $g_2(\lk_\Delta w)=2$. On the other hand, every vertex $w'\in V(\Delta-\st_\Delta u)$ is not connected to $u$, so by part 1 of Lemma \ref{lm: g_2=2 and links}, $g_2(\lk_\Delta w')=1$. Hence
    	\[ f_0(\Delta)+(d+1)=\sum_{v\in V(\Delta)} g_2(\lk_\Delta v)=2(d-1)+3g_3(\Delta)\leq 2d+4,\]
    	which implies that $f_0\leq d+3$. Hence $|V(\Delta)-V(\st_\Delta u)|=1$, and $\Delta$ is the suspension of $\lk_\Delta u$. Note that for any $2\leq i\leq d-3$, the complex $\partial \sigma^1*\partial \sigma^i*\partial\sigma^{d-i-1}$ is obtained from $\partial \sigma^i*\partial\sigma^{d-i}$ by central retriangulation along the star of a facet of $\partial\sigma^{d-i}$.
    	
    	{\bf Case 2:} $g_2(\lk_\Delta u)=0$ for some vertex $u$. As $\Delta$ is prime, every missing facet $\tau$ of $\lk_\Delta u$ is also a missing face of $\Delta$. We apply the inverse stellar retriangulation on vertex $u$ to obtain a new normal $(d-1)$-pseudomanifold $\sd_u^{-1}(\Delta)$. Since by Lemma \ref{lm: g_k inverse retriangulation}, $g_2(\sd_u^{-1}(\Delta))=g_2(\Delta)-g_1(\lk_\Delta u)\geq 0$, either $\sd_u^{-1}(\Delta)$ is a stacked $(d-1)$-sphere and we let $B:=(\lk_\Delta u)(1)$ is the union of three adjacent facets of $\sd_u^{-1}(\Delta)$, or $\sd_u^{-1}(\Delta)$ is a polytopal sphere with $g_2=1$ and we let $B$ be the union of two adjacent facets of $\sd_u^{-1}(\Delta)$. In both cases, $\Delta$ is obtained by centrally retriangulating the polytopal sphere $\sd_u^{-1}(\Delta)$ along $B$.
    \end{proof}
    
    It is left to treat the case of dimension 3. 
    \begin{theorem}\label{thm: g_2=2 and d=4}
    	Let $\Delta$ be a prime homology 3-manifold with $g_2(\Delta)=2$. Then $\Delta$ is either the octahedral 3-sphere, or the stellar subdivision of a 3-sphere with $g_2=1$ at a ridge.
    \end{theorem}
    \begin{proof}
    	Let $u$ be a vertex of minimal degree in $V(\Delta)$. Since $g_2(\Delta)=\frac{1}{2}\sum_{v\in V(\Delta)}f_0(\lk_\Delta v)-4f_0(\Delta)+10=2$, it follows that $4<\deg u\leq 7$. We have the following cases.
    	
    	{\bf Case 1:} $\deg u=5$. Then $\lk_\Delta u$ is the connected sum of two boundary complexes of 3-simplices. As before, $\sd_u^{-1}(\Delta)$ is well-defined, and $g_2(\sd_u^{-1}(\Delta))=1$. In this case $\Delta$ is the stellar subdivision of a 3-sphere with $g_2=1$ at a ridge.
    	
    	{\bf Case 2:} $\deg u=6$ or 7 and $\lk_\Delta u$ is not prime. Then $\lk_\Delta u$ is either the connected sum of three or four boundary complexes of simplices, or it is obtained by stacking over an octahedral 2-sphere. In the former case, $g_2(\sd_u^{-1}(\Delta))=g_2(\Delta)-g_1(\lk_\Delta u)=0$ or -1. So by the lower bound theorem, $\sd_u^{-1}(\Delta)$ must be stacked. Hence there exists a vertex $w\in \sd_u^{-1}(\Delta)$ of degree 4 ($w\neq u$). Then $\deg_\Delta w\leq \deg_{\sd_u^{-1}(\Delta)} w+1\leq 5$, a contradiction. If $\lk_\Delta u$ is obtained by stacking over the octahedral 2-sphere, then by applying the Swartz operation on the vertex $u$, we obtain a new complex $\Delta'$ with $g_2(\Delta')=1$ ($u'$ is the new vertex). Note that $\deg_{\Delta'} v\geq \deg_{\Delta} v-1\geq 5$ for every vertex $v\in \Delta'$, so $\Delta'$ is prime. Since $\lk_{\Delta'} u'$ is the octahedral 2-sphere, by Theorem \ref{thm: NevoNovinsky} it follows that $\Delta'$ must be the join of a 3-cycle and 4-cycle. However, $f_0(\Delta')=f_0(\Delta)\geq \deg u+1=8$, a contradiction. Hence case 2 is impossible.
    	
    	{\bf Case 3:} $\lk_\Delta u$ is the octahedral 2-sphere. Since $\st_\Delta u$ is genericially 4-rigid, $g_2(\st_\Delta u)=0$, and $g_2(\Delta[V(\st_\Delta u)])\leq g_2(\Delta)=2$, it follows that at least one pair of antipodal vertices in $\lk_\Delta u$ forms a missing edge. Let this missing edge be $\{a,b\}$. We remove $u$ and replace $\st_\Delta u$ with the 3-ball $(\lk_\Delta u \cup \{a,b\})(1)$ to obtain a new complex $\Delta'$. We have $g_2(\Delta')=1$. Moreover, $\deg_{\Delta'}v=\deg_\Delta v-1\geq 5$ if $v\in V(\lk_\Delta u)\backslash\{a,b\}$, and else $\deg_{\Delta'}v=\deg_\Delta v\geq 6$. Hence $\Delta'$ is prime. By Theorem \ref{thm: NevoNovinsky}, $\Delta'$ is the join of a 3-cycle and a $(f_0(\Delta')-3)$-cycle. Since every vertex in the $(f_0(\Delta')-3)$-cycle has degree 5, it follows that this $(f_0(\Delta')-3)$-cycle is the 4-cycle $\lk_\Delta \{a,u\}$. Hence $f_0(\Delta)=8$ and $\Delta$ is the octahedral 3-sphere.
    	
    	{\bf Case 4:} $\lk_\Delta u=(a*C)\cup (b*C)$ for a 5-cycle $C$ and two vertices $a$ and $b$. Then a similar argument as in case 3 implies that $\Delta[V(C)]$ has at least one missing edge $e$. As $C\cup e$ is the union of a 4-cycle and a 3-cycle, $(\lk_\Delta u\cup e)(1)$ is the union of a octahedral sphere $S$ and a 3-ball $B$ (which is the suspension of a triangle). We construct a new complex $\Delta'$ by removing $u$, adding a new vertex $u'$ and the edge $e$, and replacing $\st_\Delta u$ with $(S*u')\cup B$. Then $\Delta'$ is a homology 3-manifold with $g_2(\Delta')=2$. Furthermore, the degree of every vertex of $\Delta'$ is at least 6, and so $\Delta'$ is prime. By case 2 and 3, $\Delta'$ is the octahedral 3-sphere. However, this implies the vertex $a$ has $\deg_{\Delta} a=\deg_{\Delta'} a =6<7$, contradicting that $u$ is of minimal degree.
    \end{proof}
    \begin{corollary}
    	Let $d\geq 4$ and let $\Delta$ be a homology $(d-1)$-manifold with $g_2=2$. Then $\Delta$ is either the connected sum of two polytopal $(d-1)$-spheres with $g_2=1$ (not necessarily prime) given by Theorem \ref{thm: NevoNovinsky}, or it is obtained by stacking over the complexes indicated in Theorem \ref{thm:g_2=2 and d>4} and \ref{thm: g_2=2 and d=4}.
    \end{corollary}
    \begin{remark}
      Note that except for the octahedral 3-sphere (which by definition is polytopal), all prime homology $(d-1)$-spheres with $g_2=2$ are obtained by centrally retriangulating a polytopal sphere with $g_2=0$ or 1 along the star of a face or the union of three adjacent facets. Therefore by Lemma \ref{lm: other prop of central retriangulation}, all such homology spheres are polytopal. Since the connected sum of polytopes is a polytope, it follows that all homology $(d-1)$-spheres with $g_2\leq 2$ are polytopal. 
    \end{remark}
    \begin{remark}
    	In \cite[Example 6.2]{NN}, it was shown that there exist non-polytopal spheres of any dimension $\geq 5$ with $g_3=0$. The Barnette sphere $S$ is an example of a non-polytopal 3-sphere with 8 vertices and 19 facets, so $g_2(S)=f_1(S)-4f_0(S)+10=19+8-4\cdot 8+10=5$. (The construction can be found in \cite{E}.) Also in \cite{ABS}, all non-polytopal 3-spheres with nine vertices are classified and it turns out that $g_2\geq 5$ in this case as well. The minimum value of $g_2$ for non-polytopal $(d-1)$-spheres appears to be unknown at present. On the other hand, in dimension three, $g_2<10$ implies the manifold is a sphere, as was originally proved by Walkup \cite{W}. It was shown in \cite{BD1} that for all $d\geq 3$ there are triangulations of $\mathbb{RP}^2 * \mathbb{S}^{d-4}$ ($\mathbb{S}^{-1}$ is the complex consisting of the empty set) that have $g_2=3$. This raises the question of whether $\mathbb{RP}^2 * \mathbb{S}^{d-4}$ is the only non-sphere pseudomanifold with $g_2=3$ triangulations.
    \end{remark}
    \begin{remark}
    	For a simplicial ball $\Delta$, one can compute, in addition to $g_2(\Delta)$, the relative $g_2$-number $g_2(\Delta,\partial\Delta)$. It is known that $g_2(\Delta,\partial\Delta)\geq 0$. The case of equality (for $d\geq 3$) was characterized in \cite{MN2}. It would be interesting to characterize simplicial balls with $g_2(\Delta,\partial\Delta)=1$. 
    \end{remark}
    \section*{Acknowledgements}
    The author was partially supported by a graduate fellowship from NSF grant DMS-1361423. I thank Ed Swartz for pointing out that Theorem \ref{thm:g_2=2 and d>4} holds for the class of normal pseudomanifolds and showing me the results of Bagchi and Datta \cite{BD1}. I also thank Steven Klee, Eran Nevo, Isabella Novik, and the referees for many helpful comments and discussions.
     \bibliographystyle{amsplain}
    
\end{document}